\documentclass[11pt,reqno]{amsart}
\usepackage{amsmath}
\usepackage{geometry}
\usepackage{graphicx}
\usepackage{subfigure}
\usepackage{color}
\theoremstyle{plain}
\newtheorem{theorem}{Theorem}[section]
\newtheorem{lemma}[theorem]{Lemma}

\newtheorem{corollary}[theorem]{Corollary}
\theoremstyle{definition}
\newtheorem{definition}[theorem]{Definition}
\theoremstyle{remark}

\newcommand{\R}{{\mathbb R}}

\title{$n$-digit Benford distributed random variables}
\author{Azar Khosravani}
\address{Department of Science and Mathematics, Columbia College Chicago,
Chicago, IL 60605}
\email{akhosravani@colum.edu}

\author{Constantin Rasinariu}
\address{Department of Science and Mathematics, Columbia College Chicago,
Chicago, IL 60605}
\email{crasinariu@colum.edu}

\begin{document}

\begin{abstract}

The scope of this paper is twofold. First, to emphasize the use of the mod 1 map in exploring the digit distribution of random variables. We show that the well-known base- and scale-invariance of Benford variables are consequences of their associated mod 1 density functions being uniformly distributed. Second, to introduce a new concept of the $n$-digit Benford variable. Such a variable is Benford in the first $n$ digits, but it is not guaranteed to have a logarithmic distribution beyond the $n$-th digit. We conclude the paper by giving a general construction method for $n$-digit Benford variables, and provide a concrete example.
\end{abstract}


\maketitle

\section{Introduction}
In 1881, Newcomb \cite{Newcomb-1881} noticed that the first digit distribution of numerical data is not uniform but rather logarithmic. He had observed that the pages of logarithm tables were more worn out for smaller digits such as $1$ and $2$ than for larger ones, and concluded that ``[the] law of probability of the occurrence of the numbers is such that all mantissae of their logarithms are equally probable.'' He explicitly tabulated the probability of occurrence of the first and second digits. Apparently unaware of Newcomb's results, in 1938, Benford \cite{Benford-1938} found the same first digit phenomenon, and explicitly gave the formula for the probability of a number having the first digit $d$,
\begin{equation}
\label{eq:ben}
P(d)=\log\left(1+\frac{1}{d}\right), \quad d=1,2,\ldots,9~,
\end{equation}
where $\log$ is used to represent the base $10$ logarithm. He gathered  empirical evidence for formula (\ref{eq:ben}) by collecting thousands of numbers from diverse datasets, such as the area of the riverbeds, atomic weights of elements, etc.

It has been shown (see for example \cite{Hill-2011}) that the only non-trivial digit distribution that is left invariant under scale change of the underlying distribution is the Benford distribution. Scale-invariance means that a collection of data has the same digit distribution when multiplied by a constant. In his seminal paper \cite{Hill-1995}, Hill showed  that the appropriate domain for the significant digit probability is the smallest collection of positive real subsets that contains all the infinite sets of the form $\bigcup_{n=-\infty}^{\infty}[a,b) \cdot 10^n$. This set denoted by $\mathcal A$, is the smallest sigma algebra generated by $D_1$, $D_2$, \ldots, where $D_i$ is the $i$-th significant-digit function. $D_1:\mathbb{R}^{\,+}\to \{1,\ldots,9 \}$ and $D_i:\mathbb{R}^{\,+} \to \{0,1,\ldots,9 \}$ for $i \neq 1$. For example, $D_1(2.718)=2$ and $D_2(2.718)=7$. Observe that $D_i^{-1}(d) \in \mathcal A$ for all $i$ and $d$.
Within this framework, for a random variable $Y$, the Benford's first digit distribution law can be stated as
\begin{equation}
\label{e:ben}
P(D_1(Y) = d)=\log \left(1+d^{-1}\right).
\end{equation}
In general \cite{Hill-2011}, a random variable is Benford if for all $m \in \mathbb{N}$, all $d_1\in \{1,\ldots,9 \}$ and all $d_i\in \{0,1,\ldots,9 \}$ for $i>1$,
\begin{equation}
\label{e:benn}
P(D_i(Y) = d_i\text{ for }i=1,2,\ldots, m)=\log \bigg(1+ \Big(\sum_{j=1}^{m} 10^{m-j}d_j\Big)^{-1}\bigg)~.
\end{equation}

\noindent
For example, the probability of having digits $8$, $4$ and $7$ as the first, second and third significant digits, respectively, is
\begin{equation}
P(8,4,7) = \log\left(1+\frac{1}{847} \right) \simeq 0.000512~.
\end{equation}

\noindent
Hill defined base-invariance, and proved that base-invariance as well as scale-invariance imply Benford distribution of digits. Leemis et. al \cite{Leemis-2000} investigated several examples of symmetric and non-symmetric distributions that lead to Benford distributed random variables.

This paper is organized as follows. In Section \ref{mod1}, we use the mod 1 map to show that the well-known base- and scale-invariance of Benford distributed random variables are consequences of $g^{\dagger}=1$, where $g^{\dagger}$ is the associated \mbox{mod 1} density function. In Section \ref{nben}, we introduce the concept of $n$-digit Benford distributed random variables which are guaranteed to obey the log distribution in their first $n$ digits, and give a general construction method for such variables. Unless otherwise specified, throughout this paper, we assume that the base is $10$.

\section{The mod 1 map}\label{mod1}

Any positive real number $y$ can be written as $y=m\times 10^k$ for some $k \in \mathbb{Z}$, where $1\le m<10$. Then the first digits of $y$ and $m$ are the same: $D_1(y)=D_1(m)$. Let us assume that $y$ is given by a random variable $Y$ with the density function $f$. Let $X=\log Y$ be the random variable with the density function $g$.

\begin{definition}
For any real function $g:\R \to \R$ , we define $g^{\dagger}=g~
 \text{(mod 1)}$ as
$$
g^{\dagger}(x) = \begin{cases}
\sum_{k=-\infty}^{\infty}g(x+k), & \forall x\in [0,1) ,\\
0, ~&\textrm{otherwise.}
\end{cases}
$$
\end{definition}

\begin{lemma}\label{l:one}
\label{dagger}
The probability of $Y\!$  having its first digit $d$ is
\begin{equation*}
P\left(D_1(Y)=d\right) = \int_{\log(d)}^{\log(d+1)} g^{\dagger}(x)\, dx~.
\end{equation*}
\end{lemma}

\begin{proof}
Let us consider the real numbers starting with the digit $d$, i.e., $D_1(y)=d$. These numbers belong to the set
$$
S\,=\bigcup_{k=-\infty}^{\infty}[\,d\times 10^k,\,(d+1)\times 10^k).
$$
Now consider the random variable $X=\log Y$ with the density function $g$. The logarithmic function maps $S$ into $\bigcup_{k=-\infty}^{\infty}[\,\log d+k,\,\log (d+1)+k) $. This set modulo $1$ is just $[\log d,\,\log (d+1))$ as illustrated in Figure (\ref{f:mod1}) for the case $d=2$.
\begin{figure}[htbp]
\begin{center}
\includegraphics[width=0.9\textwidth]{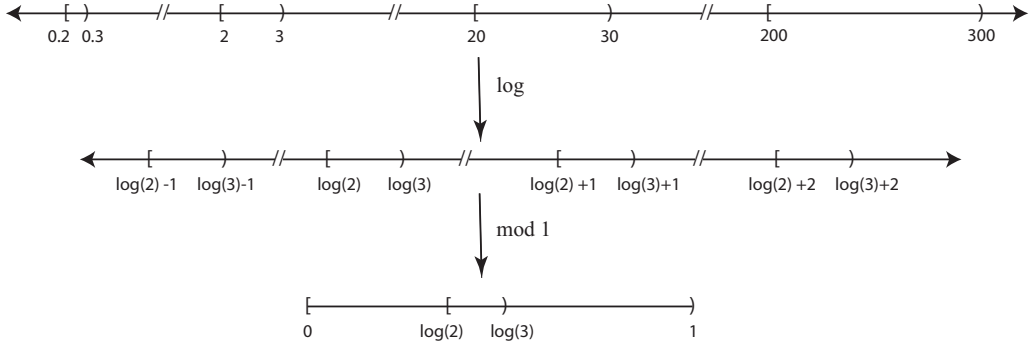}
\caption{The modulo 1 set of logarithms of positive numbers starting with $2$.}
\label{f:mod1}
\end{center}
\end{figure}

\noindent
Consequently
\begin{equation*}
\begin{split}
P\left(D_1(Y)=d\right)  &= \sum_{k=-\infty}^{\infty} P\left( d\times 10^k \le Y < (d+1)\times 10^k \right) \\
&= \sum_{k=-\infty}^{\infty}P\left(\log d +k \le X < \log (d+1) +k \right) \\
&= P\left(\log d \le X (\textrm{mod}\ 1) < \log (d+1)\right)\\
&=\int_{\log d}^{\log (d+1)} g^{\dagger}(x)\, dx~.
\end{split}
\end{equation*}

\end{proof}

Clearly, if $g^{\dagger}=1$, then one obtains Benford's law $P\left(D_1(Y)=d\right) = \log(1+1/d)$.  The uniformity of $g^{\dagger}$ resonates with Newcomb's pioneering observation in 1881 \cite{Newcomb-1881}  that the ``probability of the occurrence of the numbers is such that all mantissae of their logarithms are equally probable.''

Lemma \ref{l:one} can be readily generalized for the probability of a sequence of prescribed digits $d_1, d_2, \ldots, d_n$, where $d_1 \in \{1,2,\ldots,9\}$ and $d_i \in \{0,1,\ldots,9\}$ for $i > 1$:
\begin{equation}
\label{e:five}
P\Big(D_1(Y) = d_1,D_2(Y) = d_2,\ldots,D_n(Y)=d_n\Big)= \int_{\log(d_1+\frac{d_2}{10}+\cdots+\frac{d_n}{10^{n-1}}) }^{\log(d_1+\frac{d_2}{10}+\cdots+\frac{d_n+1}{10^{n-1}} )} g^{\dagger}(x)\,dx~.
\end{equation}

\subsection{Scale-invariance, Base-invariance, and the mod 1 map}

Scale-invariance means that a collection of data does not change its digit distribution when multiplied by a constant. For example, suppose that the prices of goods are Benford distributed, then scale-invariance implies that these prices remain Benford distributed regardless of the currency in which they are converted.

\begin{lemma}
\label{l:scale-tr}
Scaling of a random variable $Y$ is equivalent to a translation of $X = \log Y$.
\end{lemma}

\begin{proof}
Let $X$ be a random variable with the density function $g$ and let $X_1$ be the random variable generated by the translation of $X$ by $t$ units, i.e., $X_1=X+t$ and $g_1(x)=g(x-t)$, where $g_1$ is the density function of $X_1$. Define $Y=10^X$ and $Y_1=10^{X_1}$ and let $f$ and $f_1$ be their corresponding density functions. In terms of cumulative distribution functions, we have
\begin{equation*}
F_1(y)  = G_1(\log y) = G(\log y - t)
	    = G(\log \frac{y}{10^t} )
	    = F(\frac{y}{10^t} )~.
\end{equation*}
The converse is immediate.
\end{proof}

Due to modular arithmetics, a translation of $g$ will result in a wrap-around effect in $g^{\dagger}$. Hence, scaling of $Y$ induces a wrap-around of $g^{\dagger}$. For example, let  $X =\log Y$  be the random variable with the density function $g = \textrm{Triangle}(0,\frac 3 2 ,3)$, and let $X_1$ be its translation by $t$. The effect of scaling of $Y$ on $g^{\dagger}$ is shown in figure (\ref{f:wrap}).
\begin{figure}[htbp]
\begin{center}
\includegraphics[width=0.6\textwidth]{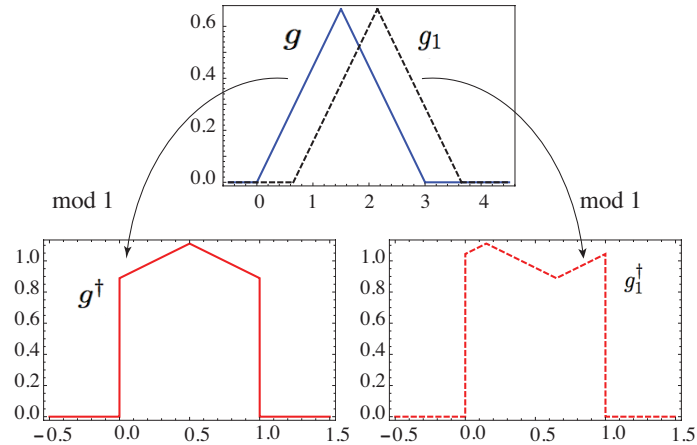}
\caption{The wrap-around effect on $g^{\dagger}$ of a translation by t=0.65 of the $g = \textrm{Triangle}(0,\frac 3 2 ,3)$.}
\label{f:wrap}
\end{center}
\end{figure}

\noindent
By Lemma (\ref{dagger}), we observe that $P\left(D_1(Y)=d\right) \ne P\left(D_1(Y_1)=d\right)$, i.e., the first digit distributions of $Y$ and $Y_1$ are not the same, indicating that $Y$ is not scale-invariant. One can see that only a uniform $g^{\dagger}$ remains unchanged under the wrap-around effect induced by a scaling of $Y$.

\begin{theorem}
\label{t:ben-scale}
Only the random variables characterized by $g^{\dagger} =1$ are scale-invariant.
\end{theorem}

\begin{proof}
By Lemma (\ref{l:scale-tr}), scaling of $Y$ yields a wrap-around effect of $g^{\dagger}$. For a Benford distribution, we want the same areas under the $g^{\dagger}$ and $g_1^{\dagger}$ over the intervals $[0, \log 2), [\log 2, \log 3),$
$\ldots,\, [\log 9, 1)$. Here $g_1^{\dagger}$ denotes the mod 1 projection of $g_1$ which is the shifted $g$.  Clearly, a uniform $g^{\dagger}$ will satisfy this constraint. Any other function when translated by an arbitrary amount, will not simultaneously keep the areas under $g^{\dagger}$  preserved over all of these intervals. Thus, the only $g^{\dagger}$ left unchanged under all translations is the uniform one.
\end{proof}

We are now arriving at two important results first proved by Hill \cite{Hill-1995} concerning scale- and base-invariance.

\begin{corollary}
Scale-invariance implies Benford's law.
\end{corollary}

\begin{proof}
In order to have scale-invariance, we must have $g^{\dagger} =1$. From Lemma (\ref{dagger}), one obtains Benford's law.
\end{proof}

Base-invariance means that for any base $b$, the probability of having $d$ as the first digit is
\begin{equation*}
P_b(D_1(Y)=d) = \log_b \left( 1 + 1/d \right),~~d=1,\ldots,b-1.
\end{equation*}

\begin{corollary}
Scale-invariance implies base-invariance.
\end{corollary}

\begin{proof}
Scale-invariance requires $g^{\dagger} =1$. Then, for an arbitrary base $b$, integrating over the logarithmic intervals $\left[\log_b d, \log_b (d+1)\right)$, we get Benford's distribution
\begin{equation}
  \label{e:base}
  P_b(D_1(Y)=d)=\log_b (1+1/d)~, \quad d\in\{1,\ldots,b-1\}~.
\end{equation}

\end{proof}

\noindent
Using the insight gained by the mod 1 map analysis, we can construct an infinite class of Benford distributed random variables with non-compact support. To this end, let us consider $X$ with the following density function
\begin{equation}
\label{e:non-c}
g(x)=\sum_{n=1}^{\infty}\frac{1}{2^n}\left[ H(x-n) - H(x-n-1) \right],
\end{equation}
as illustrated in Figure \ref{f:series}.
\begin{figure}[htbp]
\begin{center}
\includegraphics[width=0.45\textwidth]{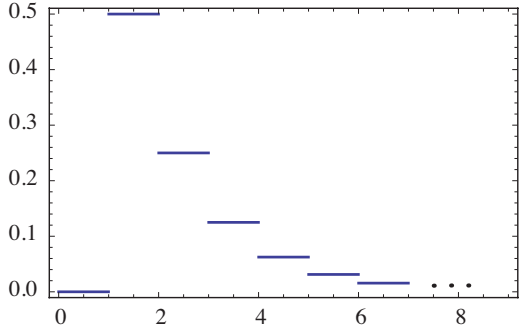}
\caption{Example of $g$ with non-compact support yielding uniform $g^{\dagger}$.}
\label{f:series}
\end{center}
\end{figure}
Here $H(x)$ is the Heaviside step function. Since the series $\sum_{n=1}^{\infty}\frac{1}{2^n}=1$, one obtains $g^{\dagger}=1$ and thus, $Y=10^X$ is Benford distributed. A similar construction can be used for any convergent series.

\section{$n$-digit Benford variables}\label{nben}

We define a random variable that has a logarithmic distribution in its first $n$ digits to be a $n$-digit Benford variable. More precisely, such a variable is Benford in the first $n$ digits, but it is not guaranteed to have a logarithmic distribution beyond the $n$-th digit.

\begin{definition}
Let $n\in\mathbb{N}$. A random variable is $n$-digit Benford if for all $d_1\in \{1,\ldots,9 \}$ and all $d_j\in \{0,1,\ldots,9 \}$, $2\le j\le n$,
\begin{equation}
\label{e:benf}
P\Big(D_1(Y) = d_1,\ldots,D_n(Y)=d_n\Big)=\log \bigg(1+ \frac{1}{10^{n-1} d_1+10^{n-2} d_2 + \cdots + d_n} \bigg).
\end{equation}
\end{definition}

To show that such variables exist, we use (\ref{e:five}) to construct a $g^{\dagger}$ that satisfies (\ref{e:benf}) for any $n\in\mathbb{N}$. That is, we search for $g^{\dagger}$ such that
\begin{equation}
  \label{e:star}
  \int_{\log(d_1+\frac{d_2}{10}+\cdots+\frac{d_n}{10^n} )}^{\log(d_1+\frac{d_2}{10}+\cdots+\frac{d_n+1}{10^n} )} g^{\dagger}(x)\,dx=\log \bigg(1+ \frac{1}{10^{n-1} d_1+10^{n-2} d_2 + \cdots + d_n} \bigg).
\end{equation}

\noindent
We proceed by the partition $\{\log 1, \log 2, \ldots, \log 10\}$ of the $[0,1)$ interval. Then an example of $g^{\dagger}$ which yields a $1$-digit Benford distribution is given by
\begin{equation}
\label{e:non-u}
g^{\dagger}(x) = \begin{cases}
\frac{\pi}{2}   \sin \left(\pi \,\frac{ x-\log k}{\log (1 + 1/k)}\right), \quad & \log k <x \le \log (k+1) ~,~~k=1,\cdots,9,\\
0, ~&\textrm{otherwise}
\end{cases}
\end{equation}
as illustrated in Figure \ref{f:sine}.
\begin{figure}[htbp]
\begin{center}
\includegraphics[width=0.45\textwidth]{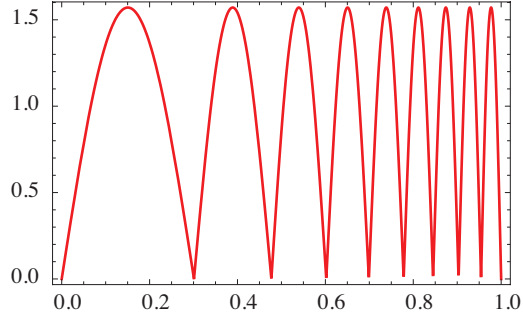}
\caption{Example of $g^{\dagger}$ yielding a $1$-digit Benford random variable.}
\label{f:sine}
\end{center}
\end{figure}
One can readily check that $\int_{\log(d)}^{\log(d+1)} g^{\dagger}(x) \, dx =\log(1+1/d)$.

\noindent
By Theorem (\ref{t:ben-scale}) any density function $g$ whose $g^{\dagger}$ is non-uniform is not scale-invariant. Thus the example of equation (\ref{e:non-u}) while yielding a $1$-digit Benford distributed variable, is not scale- nor base-invariant.

Following this idea, one can generalize the example given in (\ref{e:non-u}) in order to produce a $g^{\dagger}$ which satisfy (\ref{e:star}) with an arbitrary large $n$. For such a construction, let
\begin{equation*}
    0 = a_0 < a_1 < \cdots < a_{m-1} < a_m = 1~,~~ m = 10^n
\end{equation*}
be a partition of the interval $[0,1)$. Let $h_j : [0,1] \to [0,\infty)$ be a probability density function. For $x \in [0,1)$ such that $a_j \le x < a_{j+1}$, let us define
\begin{equation}
  \label{e:hj}
  g^{\dagger}=h_j\left(\frac{x-a_j}{a_{j+1}-a_j}\right)~.
\end{equation}
It is easy to see that for all $j$,
\begin{equation*}
  \int_{a_j}^{a_{j+1}} g^{\dagger}(x)\,dx = a_{j+1}-a_j~.
\end{equation*}
Therefore, we have that $g^{\dagger}$ behaves as the uniform density ($g(x)=1$ for all $0 \le x <1$) when integrated over intervals of the form $[a_k,a_{\ell})$ with $0\le k < \ell \le n$.

\section{Numerical modeling of a 1-digit Benford variable}
Using Mathematica \cite{Mathematica9}, we modeled a 1-digit Benford random variable by generating its discrete approximation with $100,000$ data points. As a concrete example we took $Y=10^X$ where $X$ has the density function $g(x)$ given by (\ref{e:non-u}). From $f(y)=g(\log y)/(y\ln 10)$, we obtain the density function for $Y$:
\begin{equation}
\label{e:f-non-u}
f(y) = \begin{cases}
\frac{\pi }{2 y \ln 10}   \sin \left(\frac{\pi \log \frac y k}{\log (1 + \frac 1 k)}\right), \quad &  k <y \le k+1 ~,~~k=1,\cdots,9,\\
0, ~&\textrm{otherwise .}
\end{cases}
\end{equation}
The 100-bin histogram of these data points is shown in Figure \ref{f:subfig1} along with the graph of the density function $f$ represented by dashed lines. The Mathematica generated data closely follow the actual curve of the density function, as we can see in \ref{f:subfig1}. In Figure \ref{f:subfig2}, we show the histogram of the first digit distribution of the experimental points versus the theoretical Benford probabilities.  The numerical results are recorded in table \ref{t:data}. Note the accuracy of the Mathematica generated numbers.

\begin{figure}[htb]
	\centering
	\subfigure[Density function $f(y)$ and a histogram of $100,000$ Mathematica generated numbers approximating the distribution ]{
		\includegraphics[scale =0.65] {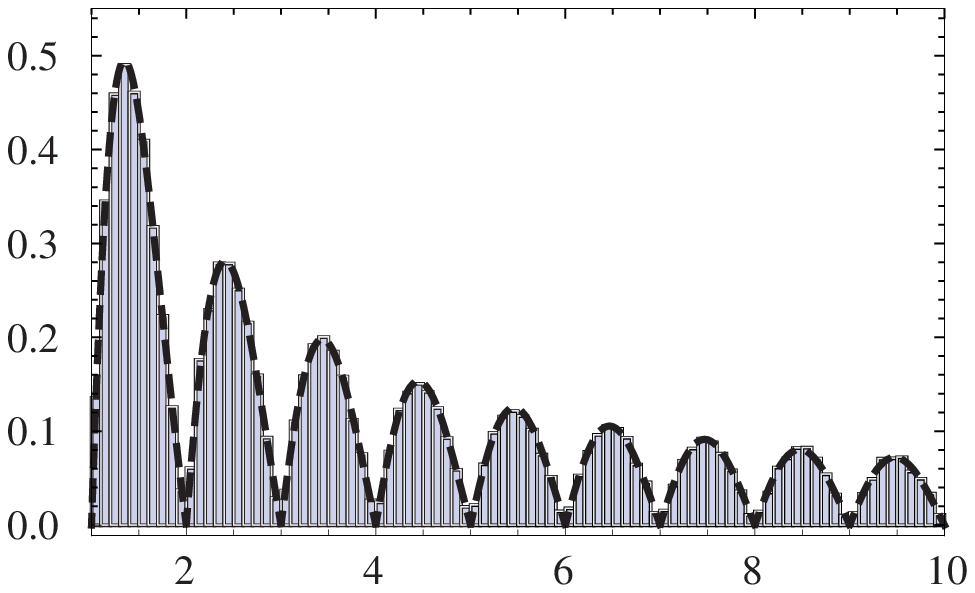}
	    \label{f:subfig1}
	}\qquad
	\subfigure[Histogram of first digits of $Y$ vs. the Benford distribution ]{
		\includegraphics[scale =0.65] {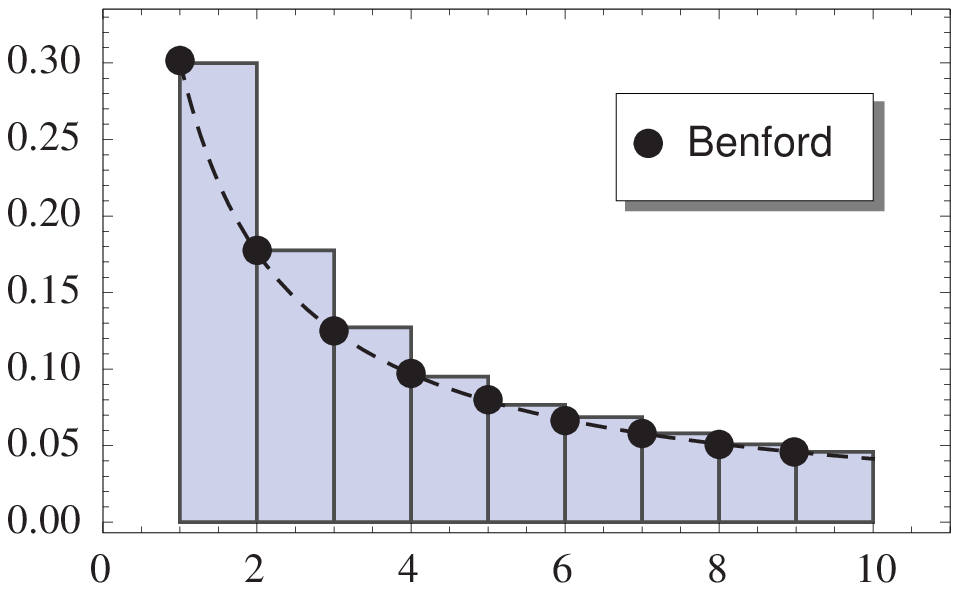}
	    \label{f:subfig2}
	}
	\caption[Optional caption for list of figures]{Example of 1-digit Benford distribution and its discrete numerical approximation along with its first digit distribution histogram.}
	\label{fig:new}
	\end{figure}

\begin{table}[htb]
{\small
\begin{tabular}{|c|c|c|c|c|c|c|c|c|c|}
  \hline
  $d$    & 1 & 2 & 3 & 4 & 5 & 6 & 7 & 8 & 9 \\ \hline
  Theoretical  & 0.3010 & 0.1761 & 0.1249 & 0.0969 & 0.0792 & 0.0669 & 0.0580 & 0.0512 & 0.0458 \\
  Mathematica & 0.3006 & 0.1775 & 0.1258 & 0.0957 & 0.0783 & 0.0671 & 0.0572 & 0.0516 & 0.0463 \\
  \hline
 \end{tabular}
} \vskip 7pt
 \caption{Theoretical vs. experimental (Mathematica generated) probabilities  for the 1-digit Benford random variable (\ref{e:f-non-u})}
  \label{t:data}
\end{table}

\section{Conclusions}

In this paper, we demonstrated the importance of mod 1 map in analyzing the digit distribution of random variables. In particular, we have shown that a uniform $g^{\dagger}$ implies scale- and base-invariance. We have introduced the concept of $n$-digit Benford, and gave a concrete example of a 1-digit Benford random variable. Furthermore, we have shown how to use the mod 1 map to construct a $n$-digit Benford variable starting with any density function.  We generated, with the help of the computer algebra system Mathematica, a $100,000$ data points discrete approximation of a 1-digit Benford variable, and found close agreement between the first digit probabilities of the model against the expected theoretical probabilities.

\section*{Acknowledgements} We are grateful to the anonymous referee for the valuable and constructive suggestions that helped improve this work. We also would like to thank Prof. Berger for very useful discussions.


\begin{thebibliography}{99}

\bibitem{Newcomb-1881}
Newcomb, S. Note on the frequency of use of the different digits in natural numbers. \emph{American Journal of Mathematics} Vol. 4, No. 1 (1881), pp. 39-40.

\bibitem{Benford-1938}
Benford, F. The law of anomalous numbers. \emph{Proceedings of the American Philosophical Society} Vol. 78, No. 4 (1938), pp. 551-572.

\bibitem{Hill-1995}
Hill, TP. Base-Invariance Implies Benford's Law. \emph{Proceedings of the American Mathematical Society}  Vol. 123, No. 3 (1995), pp. 887-895.

\bibitem{Leemis-2000}
Leemis, LM, Schmeiser, BW and Evans, DL. Survival Distributions Satisfying Benford's Law. \emph{American Statistician} Vol. 54, No. 4 (2000), pp. 236-241.

\bibitem{Hill-2011}
Berger, A and Hill, TP. A basic theory of Benford's Law. \emph{Probability Surveys} Vol. 8 (2011), pp. 1-126.

\bibitem{Mathematica9} Wolfram Research, Inc., Mathematica, Version 9, Champaign, IL (2012).

\end{thebibliography}
\end{document}